\documentclass[12pt,leqno]{amsart}

\usepackage{amsmath,amssymb,amsthm}
\usepackage{cite}
\usepackage[greek,english]{babel}
\usepackage{lmodern}
\usepackage[margin=1.45in,footskip=1.5cm]{geometry}
\usepackage{mathtools}
\usepackage{microtype}
\usepackage[T1]{fontenc}
\usepackage[utf8]{inputenc}

\makeatletter
\@namedef{subjclassname@2020}{\textup{2020} Mathematics Subject Classification}
\makeatother

\newtheorem{theorem}{Theorem}[section]
\theoremstyle{definition}
\newtheorem{definition}[theorem]{Definition}

\title[Modal sentential calculi]{On the foundations of logic and probability, I\\ Modal sentential calculi}

\author{Lo\"{i}c Cosyns}
\address{Department of Mathematical Sciences, University of Durham, UK}
\email{l.j.g.cosyns@gmail.com}
\thanks{This paper is an abridged version of the second chapter of the author’s thesis \cite{Cosyns26} presented for the degree of Doctor of Philosophy of the University of Durham (UK) and is the first of a series on the foundations of logic and probability. A companion paper on modal probability spaces is in preparation; a third paper on formal agnoiology, a field of research introduced herein, and a fourth paper on iterated modalities are under consideration.}

\subjclass[2020]{Primary 03C99}
\keywords{Foundations of modal logic, foundations of probability}

\begin{document}

\begin{abstract}
A class of models which does not require additional operators\linebreak on the language of classical sentential calculi to capture some notion of necessity, possibility, and impossibility is introduced and investigated.
\end{abstract}

\maketitle


\section{Introduction}

This study on the foundations of modal logic finds its origins in the foundations of probability. In its most general acceptation, a probability is a strictly positive, finitely additive, potentially countably additive, normed measure on a Boolean algebra of events \cite{Kappos69,Kolmogorov48,Segal54}. Remarkably, ordinary Boolean algebras constitute the algebraic semantics of classical sentential calculi. It is not clear, however, whether events truly abide by the logic of those calculi. For if, intuitively, an event is something which can occur or not, that is to say, something which is possible, then events rather abide by the logic of modal sentential calculi, an argument supported by the fact that the meet of two possible events need not be possible in general.

Modal sentential calculi have a long tradition in mathematics and philosophy. The current view, inherited from the work of Lewis \cite{Lewis59} and G\"{o}del \cite{Godel86}, hereafter, the `Lewis--G\"{o}del paradigm', consists in adding to the language of classical sentential calculi an operator $\Box$ to capture the meaning of the predicate {\it is necessary} and an operator $\Diamond$ to capture the meaning of the predicate {\it is possible}. Various axiomatic systems based on Lewis's insight and, later, that of other logicians are then postulated and investigated using different methods \cite{Kripke59,Kripke63,McKinsey41,McKinseyTarski48}. It must be emphasised that this paradigm has been tremendously successful not only in philosophical logic but also in mathematics, with the advent, for example, of the class of Boolean algebras with operators for which purely mathematical results have been established \cite{JonssonTarski51,JonssonTarski52}.

Yet no serious attempts have been made by mathematicians to develop the calculus of probability on, say, McKinsey and Tarski's closure algebras \cite{McKinseyTarski44}, which constitute the algebraic semantics of Lewis's modal system S4, or Halmos's monadic Boolean algebras \cite{Halmos56}, which constitute the algebraic semantics of Lewis's modal system S5, the general consensus being that a revamping of the foundations of probability is, to paraphrase Tarski, neither necessary nor desirable. Hence, if events truly abide by the logic of modal sentential calculi whilst forming an ordinary Boolean algebra and the foundations of probability are correct, then, perhaps, a mistake has been made and a reconsideration of the foundations of modal logic in light of the foundations of probability is in order.

This study addresses that issue by introducing a class of models for modal sentential calculi which does not require additional operators on the language of classical sentential calculi to capture some notion of necessity, possibility, and impossibility. The idea is simple and consists in defining, in a Tarskian fashion, modality directly at the level of the metalanguage instead of adding an operator $\Box$ or $\Diamond$ to the object-language. This is discussed in detail in Section 4 after some preliminary considerations in Sections 2 and 3. Note in passing that all the results presented herein are elementary, only a basic understanding of logic, algebra, topology, and measure theory being required; some experience with these matters, however, is expected of the reader.


\section{Preliminaries}

Let $\kappa\geq\omega$ be a regular cardinal \cite{ScottTarski58}. The language $\mathcal{L}$ of classical sentential calculi discussed in this study has a supply of $\kappa$ distinct sentential variables and permits the formation of disjunctions and conjunctions of all lengths $<\omega_1$. In other words, in addition to the logical operations of negation $-$ and material implication $\rightarrow$ common to all classical sentential calculi, the logical operations of countable disjunction $\bigvee$ and countable conjunction $\bigwedge$ are well defined in $\mathcal{L}$. Auxiliary parentheses ( and ), whose role is to disambiguate syntactic expressions, are also used unreservedly.

Let $\mathcal{S}$ be the least set of well-formed expressions built up from the sentential variables and logical operations of $\mathcal{L}$. Elements of $\mathcal{S}$ are called {\it formulas of the language $\mathcal{L}$ of classical sentential calculi}, or simply {\it formulas of $\mathcal{L}$}, and are denoted by $F$ with a subscript if necessary. The usual formation rules apply:\\
(i) Every sentential variable is a formula of $\mathcal{L}$;\\
(ii) If $F$ is a formula of $\mathcal{L}$, then so is $-F$;\\
(iii) If $F_1,F_2$ are formulas of $\mathcal{L}$, then so is $F_1\rightarrow F_2$;\\
(iv) If $F_0,F_1,F_2,\ldots$ are formulas of $\mathcal{L}$, then so is $\bigvee_{n\in\omega}F_n$;\\
(v) If $F_0,F_1,F_2,\ldots$ are formulas of $\mathcal{L}$, then so is $\bigwedge_{n\in\omega}F_n$.

\begin{theorem}
{\it The structure $\mathfrak{S}=(\mathcal{S};-,\rightarrow,\bigvee,\bigwedge)$, where $\mathcal{S}$ is the least set of formulas of the language $\mathcal{L}$ of classical sentential calculi, $-$ is the logical operation of negation, $\rightarrow$ is the logical operation of material implication, $\bigvee$ is the logical operation of countable disjunction, and $\bigwedge$ is the logical operation of countable conjunction, is a $\sigma$-complete abstract algebra.}
\end{theorem}

The $\sigma$-complete abstract algebra $\mathfrak{S}$ of Theorem 2.1 is called the {\it algebra of formulas of the language $\mathcal{L}$ of classical sentential calculi}. Technically, there is nothing special about $\mathfrak{S}$; it is just an abstract algebra in which countable disjunctions and countable conjunctions are well defined.

Consider now a deductive system for classical sentential calculi (e.g. Scott and Tarski \cite{ScottTarski58}). The following theorem is fundamental.

\begin{theorem}
{\it Let $\mathcal{S}$ be the least set of formulas of the language $\mathcal{L}$ of classical sentential calculi and $\equiv$ be an equivalence relation on $\mathcal{S}$, any two formulas $F_1,F_2$ of $\mathcal{L}$ being logically equivalent if and only if they materially imply each other, that is to say,
\begin{gather*}
\textstyle{F_1\equiv F_2~\text{if and only if}~F_1\rightarrow F_2~\text{and}~F_2\rightarrow F_1.}
\end{gather*}
Then the structure ${\mathfrak{S}}/{\,\equiv}=({\mathcal{S}}/{\,\equiv};-,\rightarrow,\bigvee,\bigwedge)$, where ${\mathcal{S}}/{\,\equiv}$ is the least set of formulas of the language $\mathcal{L}$ of classical sentential calculi modulo logical equivalence, $-$ is the lattice operation of complementation, $\rightarrow$ is the lattice operation of relative pseudo-complementation, $\bigvee$ is the lattice operation of countable join, and $\bigwedge$ is the lattice operation of countable meet, is a $\sigma$-complete Boolean algebra.}
\end{theorem}

The $\sigma$-complete Boolean algebra ${\mathfrak{S}}/{\,\equiv}$ of formulas of $\mathcal{L}$ modulo logical equivalence of Theorem 2.2 is called the {\it Lindenbaum--Tarski algebra of the language $\mathcal{L}$ of classical sentential calculi}. Lindenbaum--Tarski algebras ${\mathfrak{S}}/{\,\equiv}$ have interesting properties. Perhaps most importantly, since every axiom schema in the deductive system in question is a tautology, ${\mathfrak{S}}/{\,\equiv}$ is free in the class of all similar $\sigma$-complete Boolean algebras.

\begin{definition}
A Lindenbaum--Tarski algebra ${\mathfrak{S}}/{\,\equiv}$ of the language $\mathcal{L}$ of classical sentential calculi satisfies the {\it countable chain condition} (or is {\it ccc}) if every disjoint family of equivalence classes of formulas of $\mathcal{L}$ is countable.
\end{definition}

\begin{definition}
A Lindenbaum--Tarski algebra ${\mathfrak{S}}/{\,\equiv}$ of the language $\mathcal{L}$ of classical sentential calculi is {\it weakly countably distributive} if, for every double-indexed family $\{[F_{m,n}]:m,n\in\omega\times\omega\}$ of equivalence classes of formulas of $\mathcal{L}$ such that $[F_{m,n}]\leq [F_{m,n+1}]$, the relevant joins and meets exist and
\begin{gather*}
\textstyle{\bigwedge_{m\in\omega}\bigvee_{n\in\omega}[F_{m,n}]=\bigvee_{\varphi\in\omega^\omega}\bigwedge_{m\in\omega}[F_{m,\varphi(m)}],}
\end{gather*}
where $\omega^\omega$ is the set of all mappings $\varphi$ of $\omega$ into $\omega$.
\end{definition}

\begin{theorem}[von Neumann \cite{VonNeumann15}]
{\it A Lindenbaum--Tarski algebra ${\mathfrak{S}}/{\,\equiv}$ of the language $\mathcal{L}$ of classical sentential calculi carries a strictly positive, countably additive probability measure ${\bf p}$ only if it satisfies the countable chain condition and is weakly countably distributive.}
\end{theorem}

A Lindenbaum--Tarski algebra ${\mathfrak{S}}/{\,\equiv}$ of the language $\mathcal{L}$ of classical sentential calculi which is both ccc and weakly countably distributive is called {\it almost measurable}. Almost measurable Lindenbaum--Tarski algebras are denoted by ${\mathfrak{S}^\ast}/{\,\equiv}$. Of course, since every almost measurable Lindenbaum--Tarski algebra is ccc, it is complete, hence not free, no infinite free complete Boolean algebra existing in Zermelo--Fraenkel set theory with the axiom of choice \cite{Gaifman64,Hales64}.

Note that, since every free $\sigma$-complete Boolean algebra is isomorphic to the $\sigma$-field of Baire subsets of a Cantor space \cite{Rieger51}, ${\mathfrak{S}}/{\,\equiv}$ is countably distributive, hence weakly countably distributive. The countable chain condition and the property of weak countable distributivity, however, because of their necessary character for the existence of a strictly positive, countably additive measure on a $\sigma$-complete Boolean algebra, are considered here indissociable.

\begin{theorem}[Stone \cite{Stone37a}]
{\it Let ${\mathfrak{S}^\ast}/{\,\equiv}$ be an almost measurable Lindenbaum--Tarski algebra of the language $\mathcal{L}$ of classical sentential calculi, $Z$ be a set of maximal filters of ${\mathfrak{S}^\ast}/{\,\equiv}$, and $\gamma$ be a Boolean homomorphism of ${\mathfrak{S}^\ast}/{\,\equiv}$ into $\mathcal{P}(Z)$ such that, for every equivalence class $[F]$ of formulas of $\mathcal{L}$, $\gamma([F])=\{\mathbf{F}\in Z:[F]\in\mathbf{F}\}$. Then the set $Z$ equipped with the topology for which the $\gamma([F])$'s form an open base is a zero-dimensional compact Hausdorff space.}
\end{theorem}

The zero-dimensional compact Hausdorff space $Z$ of Theorem 2.6 is known as the {\it Stone space of ${\mathfrak{S}^\ast}/{\,\equiv}$}. The next two theorems characterise more precisely the Stone space of almost measurable Lindenbaum--Tarski algebras. Theorem 2.7 is an immediate consequence of the countable chain condition; Theorem 2.9, for its part, follows from both the countable chain condition and the weak countable distributivity of ${\mathfrak{S}^\ast}/{\,\equiv}$.

\begin{theorem}[Stone \cite{Stone37b}]
{\it The closure of every open set in the Stone space $Z$ of an almost measurable Lindenbaum--Tarski algebra ${\mathfrak{S}^\ast}/{\,\equiv}$ of the language $\mathcal{L}$ of classical sentential calculi is open.}
\end{theorem}

\begin{definition}
In every topological space, a set is\\
(i) {\it nowhere dense} if the interior of its closure is empty;\\
(ii) {\it meagre} if it is a countable union of nowhere dense sets.
\end{definition}

\begin{theorem}[Kelley \cite{Kelley59}]
{\it Every meagre set in the Stone space $Z$ of an almost measurable Lindenbaum--Tarski algebra ${\mathfrak{S}^\ast}/{\,\equiv}$ of the language $\mathcal{L}$ of classical sentential calculi is nowhere dense.}
\end{theorem}

A topological space in which the closure of every open set is open is said to be {\it extremally disconnected}. An extremally disconnected compact Hausdorff space is also called a {\it Stonean space}. Stonean spaces are the Stone spaces of complete Boolean algebras; Stonean spaces in which meagre sets are nowhere dense are the Stone spaces of almost measurable Boolean algebras.

\begin{definition}
A {\it Baire set} in $Z$ is an element of the least $\sigma$-field of sets containing the open-closed subsets of $Z$.
\end{definition}

The next theorem is an instance of the Loomis--Sikorski theorem.

\begin{theorem}[Loomis \cite{Loomis47}, Sikorski \cite{Sikorski48}]
{\it Every almost measurable Linden-\linebreak baum--Tarski algebra ${\mathfrak{S}^\ast}/{\,\equiv}$ of the language $\mathcal{L}$ of classical sentential calculi is isomorphic to the quotient Boolean algebra $\mathcal{B}a(Z)/\mathcal{N}$, where $\mathcal{B}a(Z)$ is the $\sigma$-field of Baire subsets of the Stone space $Z$ of ${\mathfrak{S}^\ast}/{\,\equiv}$ and $\mathcal{N}\subseteq\mathcal{B}a(Z)$ is the $\sigma$-ideal of nowhere dense sets.}
\end{theorem}


\section{Theory of truth}

\begin{definition}
A {\it model for classical sentential calculi} is an ordered system $(\mathfrak{A},Z,\mathcal{F},\xi)$, where $\mathfrak{A}={\mathfrak{S}^\ast}/{\,\equiv}$ is an almost measurable Lindenbaum--Tarski algebra of the language $\mathcal{L}$ of classical sentential calculi, $Z$ is its Stone space, $\mathcal{F}=\{\emptyset,Z\}/\mathcal{N}$ is the trivial algebra of Baire subsets of $Z$ modulo nowhere dense sets, and $\xi$ is a Boolean $\sigma$-epimorphism of $\mathfrak{A}$ into $\mathcal{F}$ such that, for every equivalence class $[F]$ of formulas of $\mathcal{L}$,
\begin{gather*}
\tag{I}
\textstyle{{\rm int}~{\rm cl}~\xi([F])^\ast=\emptyset~\text{if and only if}~\xi([F])^\ast=\emptyset,}
\end{gather*}
the open-closed representative $\xi([F])^\ast$ of the image $\xi([F])$ of $[F]$ being nowhere dense in $Z$ (`false') if and only if it is the empty set.
\end{definition}

Schema (I) probably deserves an explanation. For if it is clear that the property of `nowhere denseness' rather applies to sets than equivalence classes of sets, that is to say, sets of sets, the reader may wonder why the representative $\xi([F])^\ast$ of $\xi([F])$ is open-closed at all. The point, of course, is that $\xi([F])^\ast$ must be uniquely determined for Schema (I) to be necessary and sufficient. And, by definition, each equivalence class $\xi([F])$ of Baire subsets of $Z$ has at most one open-closed representative because $\mathcal{F}$ is a subalgebra of $\mathcal{B}a(Z)/\mathcal{N}$.

\begin{definition}
In every topological space, a set is\\
(i) {\it non-boundary} if the closure of its interior is the space \cite{Wallace39};\\
(ii) {\it comeagre} if it is a countable intersection of non-boundary sets.
\end{definition}

\begin{theorem}
{\it Let $(\mathfrak{A},Z,\mathcal{F},\xi)$ be a model for classical sentential calculi. For every equivalence class $[F]$ of formulas of $\mathcal{L}$,
\begin{gather*}
\tag{II}
\textstyle{{\rm cl}~{\rm int}~\xi([F])^\ast=Z~\text{if and only if}~\xi([F])^\ast=Z,}
\end{gather*}
the open-closed representative $\xi([F])^\ast$ of the image $\xi([F])$ of $[F]$ being non-boundary in $Z$ $($`true'\,$)$ if and only if it is the space $Z$.}
\end{theorem}

\begin{proof}
Suppose that $\xi(-[F])^\ast$ is nowhere dense in $Z$ (`false') if and only if it is the empty set, that is to say, ${\rm int}~{\rm cl}~\xi(-[F])^\ast=\emptyset$ if and only if $\xi(-[F])^\ast=\emptyset$.\linebreak It follows from the rules of the topological calculus that
\begin{align*}
\textstyle{Z\setminus{\rm int}~{\rm cl}~\xi(-[F])^\ast}&=Z~\text{if and only if}~\textstyle{Z\setminus\xi(-[F])^\ast=Z}\\
\Longleftrightarrow\hspace{0.45cm}\textstyle{{\rm cl}~Z\setminus{\rm cl}~\xi(-[F])^\ast}&=Z~\text{if and only if}~\textstyle{Z\setminus\xi(-[F])^\ast=Z}\\
\Longleftrightarrow\hspace{0.25cm}\textstyle{{\rm cl}~{\rm int}~Z\setminus\xi(-[F])^\ast}&=Z~\text{if and only if}~\textstyle{Z\setminus\xi(-[F])^\ast=Z}\\
\Longleftrightarrow\hspace{1.3cm}\textstyle{{\rm cl}~{\rm int}~\xi([F])^\ast}&=Z~\text{if and only if}~\textstyle{\xi([F])^\ast=Z.}
\end{align*}
Hence, if a sentence is true if and only if its negation is false (principle of bivalence common to all two-valued calculi), then $\xi([F])^\ast$ is non-boundary in $Z$ (`true') if and only if it is the space $Z$ as claimed.
\end{proof}

A reader familiar with model theory probably recognised in Schema (II) the epitome of a T-schema and, more generally, Tarski's semantic conception of truth \cite{Tarski83}. Hence, Schema (II) is not a formula of $\mathcal{L}$ but a formula of a stronger language which contains $\mathcal{L}$, namely, that $\mathcal{L}^E$ of mathematical English viewed as an informal version of the first-order language of Zermelo--Fraenkel set theory with the axiom of choice. In particular, for every equivalence class $[F]$ of formulas of $\mathcal{L}$, $\xi([F])^\ast$ in the definiendum of Schema (II) is the name of $[F]$ to which a predicate applies, ${\rm cl}~{\rm int}~\cdot=Z$ is the predicate itself (here: {\it is true}), and $\xi([F])^\ast=Z$, the definiens, is the translation of $[F]$ into $\mathcal{L}^E$. In Tarski's phraseology, constructions of that kind are said to be {\it formally correct}.

\begin{theorem}
{\it Let $(\mathfrak{A},Z,\mathcal{F},\xi)$ be a model for classical sentential calculi. For every equivalence class $[F]$ of formulas of $\mathcal{L}$,
\begin{gather*}
\textstyle{{\rm cl}~{\rm int}~\xi(-[F])^\ast=Z~\text{if and only if}~{\rm int}~{\rm cl}~\xi([F])^\ast=\emptyset,}
\end{gather*}
the open-closed representative $\xi(-[F])^\ast$ of the image $\xi(-[F])$ of the complement $-[F]$ of $[F]$ being non-boundary in $Z$ if and only if the open-closed representative $\xi([F])^\ast$ of the image $\xi([F])$ of $[F]$ is nowhere dense in $Z$.}
\end{theorem}

\begin{proof}
For ${\rm cl}~{\rm int}~\xi(-[F])^\ast={\rm cl}~{\rm int}~Z\setminus\xi([F])^\ast$ because $\xi$ is a Boolean homomorphism and ${\rm cl}~{\rm int}~Z\setminus\xi([F])^\ast=Z\setminus{\rm int}~{\rm cl}~\xi([F])^\ast$ by the rules of the top-\linebreak ological calculus from which the conclusion follows immediately.
\end{proof}

\begin{theorem}
{\it Let $(\mathfrak{A},Z,\mathcal{F},\xi)$ be a model for classical sentential calculi. For every equivalence class $[F_0],[F_1],[F_2],\ldots$ of formulas of $\mathcal{L}$,
\begin{gather*}
\hspace{-2cm}\textstyle{{\rm cl}~{\rm int}~\xi(\bigvee_{n\in\omega}[F_n])^\ast=Z~\text{if and only if}}\\
\hspace{2cm}\textstyle{{\rm cl}~{\rm int}~\xi([F_n])^\ast=Z~\text{for some}~n\in\omega,}
\end{gather*}
the open-closed representative $\xi(\bigvee_{n\in\omega}[F_n])^\ast$ of the image $\xi(\bigvee_{n\in\omega}[F_n])$ of the join $\bigvee_{n\in\omega}[F_n]$ being non-boundary in $Z$ if and only if at least one open-closed rep-\linebreak resentative $\xi([F_0])^\ast,\xi([F_1])^\ast,\xi([F_2])^\ast,\ldots$ of at least one image $\xi([F_0]),\xi([F_1]),\linebreak\xi([F_2]),\ldots$ of $[F_0],[F_1],[F_2],\ldots$ is non-boundary in $Z$.}
\end{theorem}

\begin{proof}
Suppose that ${\rm cl}~{\rm int}~\xi([F_n])^\ast=Z$ for some $n$. Then ${\rm cl}~{\rm int}~\xi(\bigvee_{n\in\omega}[F_n])^\ast=Z$ because $\xi([F_n])^\ast\subseteq\bigvee_{n\in\omega}\xi([F_n])^\ast=\xi(\bigvee_{n\in\omega}[F_n])^\ast$ for every $n$ and every superset of a non-boundary set is non-boundary.

Conversely, suppose that ${\rm cl}~{\rm int}~\xi(\bigvee_{n\in\omega}[F_n])^\ast=Z$. It is shown that ${\rm cl}~{\rm int}~\xi([F_n])^\ast\linebreak=Z$ for some $n$. For suppose this is not the case. Then ${\rm int}~{\rm cl}~\xi([F_n])^\ast=\emptyset$ for every $n$, whence ${\rm int}~{\rm cl}~\bigcup_{n\in\omega}\xi([F_n])^\ast=\emptyset$ by Theorem 2.9, which implies that ${\rm int}~{\rm cl}~\xi(\bigvee_{n\in\omega}[F_n])^\ast=\emptyset$ because $\bigcup_{n\in\omega}\xi([F_n])^\ast\triangle\bigvee_{n\in\omega}\xi([F_n])^\ast$ is nowhere dense in $Z$. A contradiction.
\end{proof}

Note that the Stone topology generated by the countable chain condition and the weak countable distributivity of $\mathfrak{A}$ is essential in the proof of Theorem 3.5, a countable union of nowhere dense sets being not nowhere dense in general; in fact, it can even be (everywhere) dense, the set $\mathbb{Q}$ of rationals in $\mathbb{R}$ equipped with the usual topology being a classic example. Further, a set which is not non-boundary is not nowhere dense in general; in other words, that $\xi$ is a Boolean epimorphism of $\mathfrak{A}$ into $\mathcal{F}$ is equally essential to the proof of Theorem 3.5.

The next theorem is the contradual of Theorem 2.9.

\begin{theorem}
{\it Every comeagre set in the Stone space $Z$ of an almost measurable Lindenbaum--Tarski algebra $\mathfrak{A}$ of the language $\mathcal{L}$ of classical sentential calculi is non-boundary.}
\end{theorem}

\begin{theorem}
{\it Let $(\mathfrak{A},Z,\mathcal{F},\xi)$ be a model for classical sentential calculi. For every equivalence class $[F_0],[F_1],[F_2],\ldots$ of formulas of $\mathcal{L}$,
\begin{gather*}
\hspace{-2cm}\textstyle{{\rm cl}~{\rm int}~\xi(\bigwedge_{n\in\omega}[F_n])^\ast=Z~\text{if and only if}}\\
\hspace{2cm}\textstyle{{\rm cl}~{\rm int}~\xi([F_n])^\ast=Z~\text{for every}~n\in\omega,}
\end{gather*}
the open-closed representative $\xi(\bigwedge_{n\in\omega}[F_n])^\ast$ of the image $\xi(\bigwedge_{n\in\omega}[F_n])$ of the meet $\bigwedge_{n\in\omega}[F_n]$ being non-boundary in $Z$ if and only if all open-closed representatives $\xi([F_0])^\ast,\xi([F_1])^\ast,\xi([F_2])^\ast,\ldots$ of all images $\xi([F_0]),\xi([F_1]),\xi([F_2]),\ldots$ of $[F_0],[F_1],[F_2],\ldots$ are non-boundary in $Z$.}
\end{theorem}

\begin{proof}
Suppose that ${\rm cl}~{\rm int}~\xi(\bigwedge_{n\in\omega}[F_n])^\ast=Z$. Then ${\rm cl}~{\rm int}~\xi([F])^\ast=Z$ for every $n$ because $\xi(\bigwedge_{n\in\omega}[F_n])^\ast=\bigwedge_{n\in\omega}\xi([F_n])^\ast\subseteq\xi([F_n])^\ast$ for every $n$ and every superset of a non-boundary set is non-boundary. The converse follows from Theorem 3.6, every comeagre subset of $Z$ being non-boundary.
\end{proof}

\begin{theorem}
{\it Let $(\mathfrak{A},Z,\mathcal{F},\xi)$ be a model for classical sentential calculi. For every equivalence class $[F_1],[F_2]$ of formulas of $\mathcal{L}$,
\begin{gather*}
\hspace{-2cm}\textstyle{{\rm cl}~{\rm int}~\xi([F_1]\rightarrow [F_2])^\ast=Z~\text{if and only if}}\\
\hspace{2cm}\textstyle{{\rm cl}~{\rm int}~\xi(-[F_1])^\ast=Z~\text{or}~{\rm cl}~{\rm int}~\xi([F_2])^\ast=Z,}
\end{gather*}
the open-closed representative $\xi([F_1]\rightarrow [F_2])^\ast$ of the image $\xi([F_1]\rightarrow [F_2])$ of the relative pseudo-complement $[F_1]\rightarrow [F_2]$ being non-boundary in $Z$ if and only if at least one open-closed representative $\xi(-[F_1])^\ast$ or $\xi([F_2])^\ast$ of at least one image $\xi(-[F_1])$ of $-[F_1]$ or $\xi([F_2])$ of $[F_2]$ is non-boundary in $Z$.}
\end{theorem}

\begin{proof}
As an immediate consequence of Theorem 3.5.
\end{proof}

\begin{definition}[Convention T \cite{Tarski83}]
A formally correct definition of the predicate {\it is true}, formulated in the metalanguage $\mathcal{L}^E$ of mathematical English, is {\it materially adequate} if it implies all instances of Schema (II).
\end{definition}

\begin{theorem}
{\it Let $(\mathfrak{A},Z,\mathcal{F},\xi)$ be a model for classical sentential calculi. The following schemas constitute a formally correct and materially adequate definition of the predicate \textnormal{is true:}
\begin{align*}
&\textstyle{{\rm cl}~{\rm int}~\xi([F])^\ast=Z~\text{if and only if}~\xi([F])^\ast=Z;}\\
&\textstyle{{\rm cl}~{\rm int}~\xi(-[F])^\ast=Z~\text{if and only if}~{\rm int}~{\rm cl}~\xi([F])^\ast=\emptyset;}\\
&\textstyle{{\rm cl}~{\rm int}~\xi(\bigvee_{n\in\omega}[F_n])^\ast=Z~\text{if and only if}~{\rm cl}~{\rm int}~\xi([F_n])^\ast=Z~\text{for some}~n;}\\
&\textstyle{{\rm cl}~{\rm int}~\xi(\bigwedge_{n\in\omega}[F_n])^\ast=Z~\text{if and only if}~{\rm cl}~{\rm int}~\xi([F_n])^\ast=Z~\text{for every}~n;}\\
&\textstyle{{\rm cl}~{\rm int}~\xi([F_1]\rightarrow [F_2])^\ast=Z~\text{if and only if}}\\
&\hspace{2cm}\textstyle{{\rm cl}~{\rm int}~\xi(-[F_1])^\ast=Z~\text{or}~{\rm cl}~{\rm int}~\xi([F_2])^\ast=Z,}
\end{align*}
the open-closed representative $\xi([F])^\ast$ of the image $\xi([F])$ of $[F]$ being non-boundary in $Z$ if and only if it is the space $Z$; or $\xi(-[F])^\ast$ is non-boundary in $Z$ if and only if $\xi([F])^\ast$ is nowhere dense in $Z$; or $\xi(\bigvee_{n\in\omega}[F_n])^\ast$ is non-boundary in $Z$ if and only if at least one $\xi([F_0])^\ast,\xi([F_1])^\ast,\xi([F_2])^\ast,\ldots$ is non-boundary in $Z$; or $\xi(\bigwedge_{n\in\omega}[F_n])^\ast$ is non-boundary in $Z$ if and only if all $\xi([F_0])^\ast,\xi([F_1])^\ast,\xi([F_2])^\ast,\ldots$ are non-boundary in $Z$; or $\xi([F_1]\rightarrow [F_2])^\ast$ is non-boundary in $Z$ if and only if at least one of $\xi(-[F_1])^\ast$ or $\xi([F_2])^\ast$ is non-boundary in $Z$.}
\end{theorem}

\begin{proof}
By Theorems 3.3, 3.4, 3.5, 3.7, and 3.8.
\end{proof}


\section{Modal extension}

\begin{definition}
A {\it model for modal sentential calculi} is an ordered system $(\mathfrak{A},Z,\mathcal{E},\zeta)$, where $\mathfrak{A}={\mathfrak{S}^\ast}/{\,\equiv}$ is an almost measurable Lindenbaum--Tarski algebra of the language $\mathcal{L}$ of classical sentential calculi, $Z$ is its Stone space, $\mathcal{E}=\mathcal{B}a(Z)/\mathcal{N}$ is the algebra of Baire subsets of $Z$ modulo nowhere dense sets, and $\zeta$ is a Boolean $\sigma$-isomorphism of $\mathfrak{A}$ into $\mathcal{E}$ such that, for every equivalence class $[F]$ of formulas of $\mathcal{L}$,
\begin{gather*}
\tag{III}
\textstyle{{\rm int}~{\rm cl}~\zeta([F])^\ast=\emptyset~\text{if and only if}~\zeta([F])^\ast=\emptyset,}
\end{gather*}
the open-closed representative $\zeta([F])^\ast$ of the image $\zeta([F])$ of $[F]$ being nowhere dense in $Z$ (`impossible') if and only if it is the empty set.
\end{definition}

That Schema (III) is, like Schema (I), a theorem of the metalanguage $\mathcal{L}^E$ of mathematical English if and only if $\zeta([F])^\ast$ is open-closed is clear. For suppose, to dispel any remaining doubts, that $\zeta([F])^\ast$ is closed. Then the sufficiency of Schema (III) certainly holds because the empty set is a closed, nowhere dense Baire subset of $Z$, whilst its necessity fails because the empty set is not the only subset of $Z$ which satisfies these properties, the homeomorphic copy of the Cantor set in $Z$ being a classic example.

\begin{theorem}
{\it Let $(\mathfrak{A},Z,\mathcal{E},\zeta)$ be a model for modal sentential calculi. For every equivalence class $[F]$ of formulas of $\mathcal{L}$,
\begin{gather*}
\tag{IV}
\textstyle{{\rm cl}~{\rm int}~\zeta([F])^\ast=Z~\text{if and only if}~\zeta([F])^\ast=Z,}
\end{gather*}
the open-closed representative $\zeta([F])^\ast$ of the image $\zeta([F])$ of $[F]$ being non-boundary in $Z$ $($`necessary'\,$)$ if and only if it is the space $Z$.}
\end{theorem}

\begin{proof}
Suppose that $\zeta(-[F])^\ast$ is nowhere dense in $Z$ (`impossible') if and only if it is the empty set, that is to say, ${\rm int}~{\rm cl}~\zeta(-[F])^\ast=\emptyset$ if and only if $\zeta(-[F])^\ast=\emptyset$. It follows from the rules of the topological calculus that
\begin{align*}
\textstyle{Z\setminus{\rm int}~{\rm cl}~\zeta(-[F])^\ast}&=Z~\text{if and only if}~\textstyle{Z\setminus\zeta(-[F])^\ast=Z}\\
\Longleftrightarrow\hspace{0.45cm}\textstyle{{\rm cl}~Z\setminus{\rm cl}~\zeta(-[F])^\ast}&=Z~\text{if and only if}~\textstyle{Z\setminus\zeta(-[F])^\ast=Z}\\
\Longleftrightarrow\hspace{0.25cm}\textstyle{{\rm cl}~{\rm int}~Z\setminus\zeta(-[F])^\ast}&=Z~\text{if and only if}~\textstyle{Z\setminus\zeta(-[F])^\ast=Z}\\
\Longleftrightarrow\hspace{1.3cm}\textstyle{{\rm cl}~{\rm int}~\zeta([F])^\ast}&=Z~\text{if and only if}~\textstyle{\zeta([F])^\ast=Z.}
\end{align*}
Hence, if a sentence is necessary if and only if its negation is impossible (principle of duality common to all classical modal calculi), then $\zeta([F])^\ast$ is non-boundary in $Z$ (`necessary') if and only if it is the space $Z$ as claimed.
\end{proof}

\begin{definition}
In every topological space, a set is\\
(i) {\it somewhere dense} if and only if it is not nowhere dense;\\
(ii) {\it somewhere codense} if and only if it is not non-boundary.
\end{definition}

\begin{theorem}
{\it Let $(\mathfrak{A},Z,\mathcal{E},\zeta)$ be a model for modal sentential calculi. For every equivalence class $[F]$ of formulas of $\mathcal{L}$,
\begin{gather*}
\tag{V}
\textstyle{{\rm int}~{\rm cl}~\zeta([F])^\ast\not=\emptyset~\text{if and only if}~\zeta([F])^\ast\not=\emptyset,}
\end{gather*}
the open-closed representative $\zeta([F])^\ast$ of the image $\zeta([F])$ of $[F]$ being somewhere dense in $Z$ $($`possibly true'\,$)$ if and only if it is a non-empty set.}
\end{theorem}

\begin{proof}
Recall that ${\rm int}~{\rm cl}~\zeta([F])^\ast=\emptyset$ if and only if $\zeta([F])^\ast=\emptyset$. Then ${\rm int}~{\rm cl}~\zeta([F])^\ast\linebreak\not=\emptyset$ if and only if $\zeta([F])^\ast\not=\emptyset$. Hence, if a set is somewhere dense in $Z$ if and only if it is not nowhere dense and a sentence is possibly true if and only if it is not impossible, then $\zeta([F])^\ast$ is somewhere dense in $Z$ (`possibly true') if and only if it is a non-empty set as required.
\end{proof}

\begin{theorem}
{\it Let $(\mathfrak{A},Z,\mathcal{E},\zeta)$ be a model for modal sentential calculi. For every equivalence class $[F]$ of formulas of $\mathcal{L}$,
\begin{gather*}
\tag{VI}
\textstyle{{\rm cl}~{\rm int}~\zeta([F])^\ast\not=Z~\text{if and only if}~\zeta([F])^\ast\not=Z,}
\end{gather*}
the open-closed representative $\zeta([F])^\ast$ of the image $\zeta([F])$ of $[F]$ being somewhere codense in $Z$ $($`possibly false'\,$)$ if and only if it is not the space $Z$.}
\end{theorem}

\begin{proof}
Recall that ${\rm cl}~{\rm int}~\zeta([F])^\ast=Z$ if and only if $\zeta([F])^\ast=Z$. Then ${\rm cl}~{\rm int}~\zeta([F])^\ast\linebreak\not=Z$ if and only if $\zeta([F])^\ast\not=Z$. Hence, if a set is somewhere codense in $Z$ if and only if it is not non-boundary and a sentence is possibly false if and only if it is not necessary, then $\zeta([F])^\ast$ is somewhere codense in $Z$ (`possibly false') if and only if it is not the space $Z$ as required.
\end{proof}

Schemas (III)--(VI) are noteworthy. For, contrary to the Lewis--G\"{o}del paradigm \cite{Godel86,Lewis59} currently operating in full force, no additional operators on the language $\mathcal{L}$ of classical sentential calculi, that is to say, no additional operators on the Lindenbaum--Tarski algebra $\mathfrak{A}$, was required to capture some notion of necessity, possibility, and impossibility. This follows, of course, from the fact that the aforementioned modes of truth, in a Tarskian fashion, were defined at the level of the metalanguage $\mathcal{L}^E$ of mathematical English instead.

\begin{theorem}
{\it Let $(\mathfrak{A},Z,\mathcal{E},\zeta)$ be a model for modal sentential calculi. For every equivalence class $[F]$ of formulas of $\mathcal{L}$,
\begin{gather*}
\textstyle{{\rm cl}~{\rm int}~\xi([F])^\ast=Z~\text{whenever}~{\rm cl}~{\rm int}~\zeta([F])^\ast=Z,}
\end{gather*}
the open-closed representative $\xi([F])^\ast$ of the image $\xi([F])$ of $[F]$ being non-boundary in $Z$ whenever the open-closed representative $\zeta([F])^\ast$ of the image $\zeta([F])$ of $[F]$ is non-boundary in $Z$, but not conversely.}
\end{theorem}

\begin{proof}
Recall that $\zeta$ is a Boolean $\sigma$-isomorphism of $\mathfrak{A}$ into $\mathcal{E}$ and that $\xi$ is a Boolean $\sigma$-epimorphism of $\mathfrak{A}$ into $\mathcal{F}\subseteq\mathcal{E}$. Set $\theta=\xi\circ\zeta^{-1}$. Then the composition $\theta\circ\zeta=(\xi\circ\zeta^{-1})\circ\zeta=\xi\circ(\zeta\circ\zeta^{-1})=\xi$ is a Boolean $\sigma$-epimorphism of $\mathfrak{A}$ into $\mathcal{F}$ because every Boolean isomorphism is a Boolean epimorphism.

Suppose now that ${\rm cl}~{\rm int}~\zeta([F])^\ast=Z$, that is to say, by Theorem 4.2, $\zeta([F])^\ast=Z$. Then $\xi([F])^\ast=(\theta\circ\zeta)([F])^\ast=\theta(\zeta([F])^\ast)=\theta(Z)=Z$, whence ${\rm cl}~{\rm int}~\xi([F])^\ast\linebreak=Z$ by Theorem 3.3 as required. That the converse fails in general follows from the fact that $\theta$, as a Boolean epimorphism, is not invertible.
\end{proof}

\begin{theorem}
{\it Let $(\mathfrak{A},Z,\mathcal{E},\zeta)$ be a model for modal sentential calculi. For every equivalence class $[F]$ of formulas of $\mathcal{L}$,
\begin{gather*}
\textstyle{{\rm int}~{\rm cl}~\zeta([F])^\ast\not=\emptyset~\text{whenever}~{\rm cl}~{\rm int}~\zeta([F])^\ast=Z,}
\end{gather*}
the open-closed representative $\zeta([F])^\ast$ of the image $\zeta([F])$ of $[F]$ being somewhere dense in $Z$ whenever the open-closed representative $\zeta([F])^\ast$ of the image $\zeta([F])$ of $[F]$ is non-boundary in $Z$, but not conversely.}
\end{theorem}

\begin{proof}
Indeed, it follows immediately from the rules of the topological calcu-\linebreak lus that ${\rm cl}~{\rm int}~\zeta([F])^\ast=Z$ if and only if ${\rm int}~{\rm cl}~{\rm int}~\zeta([F])^\ast={\rm int}~Z$ if and only if ${\rm int}~{\rm cl}~\zeta([F])^\ast=Z$ because $\zeta([F])^\ast$ is open(-closed) and ${\rm int}~Z=Z$. This implies that ${\rm int}~{\rm cl}~\zeta([F])^\ast\not=\emptyset$ as required. That the converse fails is obvious,\linebreak a somewhere dense set being not non-boundary in general.
\end{proof}

\begin{theorem}
{\it Let $(\mathfrak{A},Z,\mathcal{E},\zeta)$ be a model for modal sentential calculi. For every equivalence class $[F]$ of formulas of $\mathcal{L}$,
\begin{gather*}
\textstyle{{\rm int}~{\rm cl}~\zeta([F])^\ast\not=\emptyset~\text{whenever}~{\rm cl}~{\rm int}~\xi([F])^\ast=Z,}
\end{gather*}
the open-closed representative $\zeta([F])^\ast$ of the image $\zeta([F])$ of $[F]$ being somewhere dense in $Z$ whenever the open-closed representative $\xi([F])^\ast$ of the image $\xi([F])$ of $[F]$ is non-boundary in $Z$, but not conversely.}
\end{theorem}

\begin{proof}
Recall that $\theta\circ\zeta=\xi$ is a Boolean $\sigma$-epimorphism of $\mathfrak{A}$ into $\mathcal{F}$, and suppose that ${\rm cl}~{\rm int}~\xi([F])^\ast=Z$, that is to say, by Theorem 3.3, $\xi([F])^\ast=Z$. Then $Z=(\theta\circ\zeta)([F])^\ast=\theta(\zeta([F])^\ast)$, which immediately implies that $\zeta([F])^\ast\linebreak\not=\emptyset$ because $\theta(\zeta([F])^\ast)=\emptyset$ whenever $\zeta([F])^\ast=\emptyset$. Hence, ${\rm int}~{\rm cl}~\zeta([F])^\ast\not=\emptyset$\linebreak by Theorem 4.4 whenever ${\rm cl}~{\rm int}~\xi([F])^\ast=Z$ as claimed. That the converse fails in general follows, once again, from the fact that $\theta$ is not invertible.
\end{proof}

It follows from Theorems 4.6, 4.7, and 4.8 that a sentence which is necessary is true, but not conversely; that a sentence which is necessary is possible (i.e. possibly true), but not conversely; and that a sentence which is true is possible, but not conversely. These results conform with the intuition and capture the basic properties of necessity and possibility.

\begin{theorem}
{\it Let $(\mathfrak{A},Z,\mathcal{E},\zeta)$ be a model for modal sentential calculi. For every equivalence class $[F]$ of formulas of $\mathcal{L}$,
\begin{gather*}
\textstyle{{\rm cl}~{\rm int}~\zeta(-[F])^\ast=Z~\text{if and only if}~{\rm int}~{\rm cl}~\zeta([F])^\ast=\emptyset,}
\end{gather*}
the open-closed representative $\zeta(-[F])^\ast$ of the image $\zeta(-[F])$ of the complement $-[F]$ of $[F]$ being non-boundary in $Z$ if and only if the open-closed representative $\zeta([F])^\ast$ of the image $\zeta([F])$ of $[F]$ is nowhere dense in $Z$.}
\end{theorem}

\begin{proof}
For ${\rm cl}~{\rm int}~\zeta(-[F])^\ast={\rm cl}~{\rm int}~Z\setminus\zeta([F])^\ast$ because $\zeta$ is a Boolean homomorphism and ${\rm cl}~{\rm int}~Z\setminus\zeta([F])^\ast=Z\setminus{\rm int}~{\rm cl}~\zeta([F])^\ast$ by the rules of the top-\linebreak ological calculus from which the conclusion follows immediately.
\end{proof}

\begin{theorem}
{\it Let $(\mathfrak{A},Z,\mathcal{E},\zeta)$ be a model for modal sentential calculi. For every equivalence class $[F_0],[F_1],[F_2],\ldots$ of formulas of $\mathcal{L}$,
\begin{gather*}
\hspace{-2cm}\textstyle{{\rm cl}~{\rm int}~\zeta(\bigvee_{n\in\omega}[F_n])^\ast=Z~\text{whenever}}\\
\hspace{2cm}\textstyle{{\rm cl}~{\rm int}~\zeta([F_n])^\ast=Z~\text{for some}~n\in\omega,}
\end{gather*}
the open-closed representative $\zeta(\bigvee_{n\in\omega}[F_n])^\ast$ of the image $\zeta(\bigvee_{n\in\omega}[F_n])$ of the join $\bigvee_{n\in\omega}[F_n]$ being non-boundary in $Z$ whenever at least one open-closed representative $\zeta([F_0])^\ast,\zeta([F_1])^\ast,\zeta([F_2])^\ast,\ldots$ of at least one image $\zeta([F_0]),\zeta([F_1]),\zeta([F_2]),\linebreak\ldots$ of $[F_0],[F_1],[F_2],\ldots$ is non-boundary in $Z$.}
\end{theorem}

\begin{proof}
Suppose that ${\rm cl}~{\rm int}~\zeta([F_n])^\ast=Z$ for some $n$. Then ${\rm cl}~{\rm int}~\zeta(\bigvee_{n\in\omega}[F_n])^\ast=Z$ because $\zeta([F_n])^\ast\subseteq\bigvee_{n\in\omega}\zeta([F_n])^\ast=\zeta(\bigvee_{n\in\omega}[F_n])^\ast$ for every $n$ and every superset of a non-boundary set is non-boundary.

To see that the necessity fails in general, suppose that $\zeta([F])^\ast$ is neither non-boundary nor nowhere dense in $Z$. Then $\zeta(-[F])^\ast$ is, by Theorem 4.9, neither nowhere dense nor non-boundary in $Z$. Yet $\zeta([F])^\ast\cup\zeta(-[F])^\ast=\zeta([F]\vee -[F])^\ast\linebreak=Z$ is certainly non-boundary in $Z$.
\end{proof}

\begin{theorem}
{\it Let $(\mathfrak{A},Z,\mathcal{E},\zeta)$ be a model for modal sentential calculi. For every equivalence class $[F_0],[F_1],[F_2],\ldots$ of formulas of $\mathcal{L}$
\begin{gather*}
\hspace{-2cm}\textstyle{{\rm cl}~{\rm int}~\zeta(\bigwedge_{n\in\omega}[F_n])^\ast=Z~\text{if and only if}}\\
\hspace{2cm}\textstyle{{\rm cl}~{\rm int}~\zeta([F_n])^\ast=Z~\text{for every}~n\in\omega,}
\end{gather*}
the open-closed representative $\zeta(\bigwedge_{n\in\omega}[F_n])^\ast$ of the image $\zeta(\bigwedge_{n\in\omega}[F_n])$ of the meet $\bigwedge_{n\in\omega}[F_n]$ being non-boundary in $Z$ if and only if all open-closed representatives $\zeta([F_0])^\ast,\zeta([F_1])^\ast,\zeta([F_2])^\ast,\ldots$ of all images $\zeta([F_0]),\zeta([F_1]),\zeta([F_2]),\ldots$ of $[F_0],[F_1],[F_2],\ldots$ are non-boundary in $Z$.}
\end{theorem}

\begin{proof}
Suppose that ${\rm cl}~{\rm int}~\zeta(\bigwedge_{n\in\omega}[F_n])^\ast=Z$. Then ${\rm cl}~{\rm int}~\zeta([F])^\ast=Z$ for every $n$ because $\zeta(\bigwedge_{n\in\omega}[F_n])^\ast=\bigwedge_{n\in\omega}\zeta([F_n])^\ast\subseteq\zeta([F_n])^\ast$ for every $n$ and every superset of a non-boundary set is non-boundary. The converse follows from Theorem 3.6, every comeagre subset of $Z$, that is to say, every countable intersection of non-boundary subsets of $Z$, being non-boundary.
\end{proof}

\begin{theorem}
{\it Let $(\mathfrak{A},Z,\mathcal{E},\zeta)$ be a model for modal sentential calculi. For every equivalence class $[F_1],[F_2]$ of formulas of $\mathcal{L}$,
\begin{gather*}
\hspace{-2cm}\textstyle{{\rm cl}~{\rm int}~\zeta([F_1]\rightarrow [F_2])^\ast=Z~\text{whenever}}\\
\hspace{2cm}\textstyle{{\rm cl}~{\rm int}~\zeta(-[F_1])^\ast=Z~\text{or}~{\rm cl}~{\rm int}~\zeta([F_2])^\ast=Z,}
\end{gather*}
the open-closed representative $\zeta([F_1]\rightarrow [F_2])^\ast$ of the image $\zeta([F_1]\rightarrow [F_2])$ of the relative pseudo-complement $[F_1]\rightarrow [F_2]$ being non-boundary in $Z$ whenever at least one open-closed representative $\zeta(-[F_1])^\ast$ or $\zeta([F_2])^\ast$ of at least one image $\zeta(-[F_1])$ of $-[F_1]$ or $\zeta([F_2])$ of $[F_2]$ is non-boundary in $Z$.}
\end{theorem}

\begin{proof}
As an immediate consequence of Theorem 4.10.

In fact, to see that the necessity fails in general, suppose that $\zeta([F])^\ast$ is neither non-boundary nor nowhere dense in $Z$. Then $\zeta(-[F])^\ast$ is neither nowhere dense nor non-boundary in $Z$. Yet $\zeta(-[F])^\ast\cup\zeta([F])^\ast=\zeta(-[F]\vee [F])^\ast=\zeta([F]\rightarrow [F])^\ast$ is non-boundary in $Z$ because $[F]\rightarrow [F]$ is a tautology of $\mathcal{L}$.
\end{proof}

\begin{definition}
A formally correct definition of the predicate {\it is necessary}, formulated in the metalanguage $\mathcal{L}^E$ of mathematical English, is {\it weakly materially adequate} if it implies some instances of Schemas (IV).
\end{definition}

\begin{theorem}
{\it Let $(\mathfrak{A},Z,\mathcal{E},\zeta)$ be a model for modal sentential calculi. The following schemas constitute a formally correct and weakly materially adequate definition of the predicate \textup{is necessary:}
\begin{align*}
&\textstyle{{\rm cl}~{\rm int}~\zeta([F])^\ast=Z~\text{if and only if}~\zeta([F])^\ast=Z;}\\
&\textstyle{{\rm cl}~{\rm int}~\zeta(-[F])^\ast=Z~\text{if and only if}~{\rm int}~{\rm cl}~\zeta([F])^\ast=\emptyset;}\\
&\textstyle{{\rm cl}~{\rm int}~\zeta(\bigvee_{n\in\omega}[F_n])^\ast=Z~\text{whenever}~{\rm cl}~{\rm int}~\zeta([F_n])^\ast=Z~\text{for some}~n;}\\
&\textstyle{{\rm cl}~{\rm int}~\zeta(\bigwedge_{n\in\omega}[F_n])^\ast=Z~\text{if and only if}~{\rm cl}~{\rm int}~\zeta([F_n])^\ast=Z~\text{for every}~n;}\\
&\textstyle{{\rm cl}~{\rm int}~\zeta([F_1]\rightarrow [F_2])^\ast=Z~\text{whenever}}\\
&\hspace{2cm}\textstyle{{\rm cl}~{\rm int}~\zeta(-[F_1])^\ast=Z~\text{or}~{\rm cl}~{\rm int}~\zeta([F_2])^\ast=Z,}
\end{align*}
the open-closed representative $\zeta([F])^\ast$ of the image $\zeta([F])$ of $[F]$ being non-boundary in $Z$ if and only if it is the space $Z$; or $\zeta(-[F])^\ast$ is non-boundary in $Z$ if and only if $\zeta([F])^\ast$ is nowhere dense in $Z$; or $\zeta(\bigvee_{n\in\omega}[F_n])^\ast$ is non-boundary in $Z$ whenever at least one $\zeta([F_0])^\ast,\zeta([F_1])^\ast,\zeta([F_2])^\ast,\ldots$ is non-boundary in $Z$; or $\zeta(\bigwedge_{n\in\omega}[F_n])^\ast$ is non-boundary in $Z$ if and only if all $\zeta([F_0])^\ast,\zeta([F_1])^\ast,\zeta([F_2])^\ast,\ldots$ are non-boundary in $Z$; or $\zeta([F_1]\rightarrow [F_2])^\ast$ is non-boundary in $Z$ whenever at least one of $\zeta(-[F_1])^\ast$ or $\zeta([F_2])^\ast$ is non-boundary in $Z$.}
\end{theorem}

\begin{proof}
By Theorems 4.2, 4.9, 4.10, 4.11, and 4.12.
\end{proof}

The definition of necessity formulated in Theorem 4.14 conforms with the intuition. Hence, it follows from Theorem 4.9 that a sentence is necessary if and only if its negation is impossible (i.e. not possibly true) and from Theorem 4.11 that a conjunction of sentences is necessary if and only if all the sentences forming the conjunction are necessary. By Theorem 4.10, however, a disjunction of sentences may be necessary although none of the sentences forming the disjunction are necessary, a result naturally extending, {\it mutatis mutandis}, to material implication by Theorem 4.12.

\begin{theorem}
{\it Let $(\mathfrak{A},Z,\mathcal{E},\zeta)$ be a model for modal sentential calculi. For every equivalence class $[F]$ of formulas of $\mathcal{L}$,
\begin{gather*}
\textstyle{{\rm int}~{\rm cl}~\zeta(-[F])^\ast\not=\emptyset~\text{if and only if}~{\rm cl}~{\rm int}~\zeta([F])^\ast\not=Z,}
\end{gather*}
the open-closed representative $\zeta(-[F])^\ast$ of the image $\zeta(-[F])$ of the complement $-[F]$ of $[F]$ being somewhere dense in $Z$ if and only if the open-closed representative $\zeta([F])^\ast$ of the image $\zeta([F])$ of $[F]$ is somewhere codense in $Z$.}
\end{theorem}

\begin{proof}
For ${\rm int}~{\rm cl}~\zeta(-[F])^\ast={\rm int}~{\rm cl}~Z\setminus\zeta([F])^\ast$ because $\zeta$ is a Boolean homomorphism and ${\rm int}~{\rm cl}~Z\setminus\zeta([F])^\ast=Z\setminus{\rm cl}~{\rm int}~\zeta([F])^\ast$ by the rules of the top-\linebreak ological calculus from which the conclusion follows immediately.
\end{proof}

\begin{theorem}
{\it Let $(\mathfrak{A},Z,\mathcal{E},\zeta)$ be a model for modal sentential calculi. For every equivalence class $[F_0],[F_1],[F_2],\ldots$ of formulas of $\mathcal{L}$,
\begin{gather*}
\hspace{-2cm}\textstyle{{\rm int}~{\rm cl}~\zeta(\bigvee_{n\in\omega}[F_n])^\ast\not=\emptyset~\text{if and only if}}\\
\hspace{2cm}\textstyle{{\rm int}~{\rm cl}~\zeta([F_n])^\ast\not=\emptyset~\text{for some}~n\in\omega,}
\end{gather*}
the open-closed representative $\zeta(\bigvee_{n\in\omega}[F_n])^\ast$ of the image $\zeta(\bigvee_{n\in\omega}[F_n])$ of the join $\bigvee_{n\in\omega}[F_n]$ being somewhere dense in $Z$ if and only if at least one open-closed representative $\zeta([F_0])^\ast,\zeta([F_1])^\ast,\zeta([F_2])^\ast,\ldots$ of at least one image $\zeta([F_0]),\zeta([F_1]),\linebreak\zeta([F_2]),\ldots$ of $[F_0],[F_1],[F_2],\ldots$ is somewhere dense in $Z$.}
\end{theorem}

\begin{proof}
Suppose that ${\rm int}~{\rm cl}~\zeta([F_n])^\ast\not=\emptyset$ for some $n$. Then ${\rm int}~{\rm cl}~\zeta(\bigvee_{n\in\omega}[F_n])^\ast\not=\emptyset$ because $\zeta([F_n])^\ast\subseteq\bigvee_{n\in\omega}\zeta([F_n])^\ast=\zeta(\bigvee_{n\in\omega}[F_n])^\ast$ for every $n$ and every superset of a somewhere dense set is somewhere dense.

Conversely, suppose that ${\rm int}~{\rm cl}~\zeta(\bigvee_{n\in\omega}[F_n])^\ast\not=\emptyset$. It is shown that ${\rm int}~{\rm cl}~\zeta([F_n])^\ast\linebreak\not=\emptyset$ for some $n$. For suppose this is not the case. Then ${\rm int}~{\rm cl}~\zeta([F_n])^\ast=\emptyset$ for every $n$, whence ${\rm int}~{\rm cl}~\bigcup_{n\in\omega}\zeta([F_n])^\ast=\emptyset$ by Theorem 2.9, which implies that ${\rm int}~{\rm cl}~\zeta(\bigvee_{n\in\omega}[F_n])^\ast=\emptyset$ because $\bigcup_{n\in\omega}\zeta([F_n])^\ast\triangle\bigvee_{n\in\omega}\zeta([F_n])^\ast$ is nowhere dense in $Z$. A contradiction.
\end{proof}

\begin{theorem}
{\it Let $(\mathfrak{A},Z,\mathcal{E},\zeta)$ be a model for modal sentential calculi. For every equivalence class $[F_0],[F_1],[F_2],\ldots$ of formulas of $\mathcal{L}$,
\begin{gather*}
\hspace{-2cm}\textstyle{{\rm int}~{\rm cl}~\zeta([F_n])^\ast\not=\emptyset~\text{for every}~n\in\omega}\\
\hspace{2cm}\textstyle{\text{whenever}~{\rm int}~{\rm cl}~\zeta(\bigwedge_{n\in\omega}[F_n])^\ast\not=\emptyset,}
\end{gather*}
all open-closed representatives $\zeta([F_0])^\ast,\zeta([F_1])^\ast,\zeta([F_2])^\ast,\ldots$ of all images $\zeta([F_0]),\linebreak\zeta([F_1]),\zeta([F_2]),\ldots$ of $[F_0],[F_1],[F_2],\ldots$ being somewhere dense in $Z$ whenever the open-closed representative $\zeta(\bigwedge_{n\in\omega}[F_n])^\ast$ of the image $\zeta(\bigwedge_{n\in\omega}[F_n])$ of the meet $\bigwedge_{n\in\omega}[F_n]$ is somewhere dense in $Z$.}
\end{theorem}

\begin{proof}
Suppose that ${\rm int}~{\rm cl}~\zeta(\bigwedge_{n\in\omega}[F_n])^\ast\not=\emptyset$. Then ${\rm int}~{\rm cl}~\zeta([F])^\ast\not=\emptyset$ for every $n$ because $\zeta(\bigwedge_{n\in\omega}[F_n])^\ast=\bigwedge_{n\in\omega}\zeta([F_n])^\ast\subseteq\zeta([F_n])^\ast$ for every $n$ and every superset of a somewhere dense set is somewhere dense.

To see that the necessity fails in general, suppose that $\zeta([F_1])^\ast,\zeta([F_2])^\ast$ are disjoint somewhere dense subsets of $Z$. Then $\zeta([F_1])^\ast\cap\zeta([F_2])^\ast=\zeta([F_1]\wedge [F_2])^\ast\linebreak=\emptyset$ is nowhere dense in $Z$.
\end{proof}

\begin{theorem}
{\it Let $(\mathfrak{A},Z,\mathcal{E},\zeta)$ be a model for modal sentential calculi. For every equivalence class $[F_1],[F_2]$ of formulas of $\mathcal{L}$,
\begin{gather*}
\hspace{-2cm}\textstyle{{\rm int}~{\rm cl}~\zeta([F_1]\rightarrow [F_2])^\ast\not=\emptyset~\text{if and only if}}\\
\hspace{2cm}\textstyle{{\rm int}~{\rm cl}~\zeta(-[F_1])^\ast\not=\emptyset~\text{or}~{\rm int}~{\rm cl}~\zeta([F_2])^\ast\not=\emptyset,}
\end{gather*}
the open-closed representative $\zeta([F_1]\rightarrow [F_2])^\ast$ of the image $\zeta([F_1]\rightarrow [F_2])$ of the relative pseudo-complement $[F_1]\rightarrow [F_2]$ being somewhere dense in $Z$ if and only if at least one open-closed representative $\zeta(-[F_1])^\ast$ or $\zeta([F_2])^\ast$ of at least one image $\zeta(-[F_1])$ of $-[F_1]$ or $\zeta([F_2])$ of $[F_2]$ is somewhere dense in $Z$.}
\end{theorem}

\begin{proof}
As an immediate consequence of Theorem 4.16.
\end{proof}

\begin{definition}
A formally correct definition of the predicate {\it is possible}, formulated in the metalanguage $\mathcal{L}^E$ of mathematical English, is {\it weakly materially adequate} if it is implied by some instances of Schema (V).
\end{definition}

\begin{theorem}
{\it Let $(\mathfrak{A},Z,\mathcal{E},\zeta)$ be a model for modal sentential calculi. The following schemas constitute a formally correct and weakly materially adequate definition of the predicate \textup{is possible:}
\begin{align*}
&\textstyle{{\rm int}~{\rm cl}~\zeta([F])^\ast\not=\emptyset~\text{if and only if}~\zeta([F])^\ast\not=\emptyset;}\\
&\textstyle{{\rm int}~{\rm cl}~\zeta(-[F])^\ast\not=\emptyset~\text{if and only if}~{\rm cl}~{\rm int}~\zeta([F])^\ast\not=Z;}\\
&\textstyle{{\rm int}~{\rm cl}~\zeta(\bigvee_{n\in\omega}[F_n])^\ast\not=\emptyset~\text{if and only if}~{\rm int}~{\rm cl}~\zeta([F_n])^\ast\not=\emptyset~\text{for some}~n;}\\
&\textstyle{{\rm int}~{\rm cl}~\zeta([F_n])^\ast\not=\emptyset~\text{for every}~n~\text{whenever}~{\rm int}~{\rm cl}~\zeta(\bigwedge_{n\in\omega}[F_n])^\ast\not=\emptyset;}\\
&\textstyle{{\rm int}~{\rm cl}~\zeta([F_1]\rightarrow [F_2])^\ast\not=\emptyset~\text{if and only if}}\\
&\hspace{2cm}\textstyle{{\rm int}~{\rm cl}~\zeta(-[F_1])^\ast\not=\emptyset~\text{or}~{\rm int}~{\rm cl}~\zeta([F_2])^\ast\not=\emptyset,}
\end{align*}
the open-closed representative $\zeta([F])^\ast$ of the image $\zeta([F])$ of $[F]$ being somewhere dense in $Z$ if and only if it is a non-empty set; or $\zeta(-[F])^\ast$ is somewhere dense in $Z$ if and only if $\zeta([F])^\ast$ is somewhere codense in $Z$; or $\zeta(\bigvee_{n\in\omega}[F_n])^\ast$ is somewhere dense in $Z$ if and only if at least one $\zeta([F_0])^\ast,\zeta([F_1])^\ast,\zeta([F_2])^\ast,\ldots$ is somewhere dense in $Z$; or all $\zeta([F_0])^\ast,\zeta([F_1])^\ast,\zeta([F_2])^\ast,\ldots$ are somewhere dense in $Z$ whenever $\zeta(\bigwedge_{n\in\omega}[F_n])^\ast$ is somewhere dense in $Z$; or $\zeta([F_1]\rightarrow [F_2])^\ast$ is somewhere dense in $Z$ if and only if at least one of $\zeta(-[F_1])^\ast$ or $\zeta([F_2])^\ast$ is somewhere dense in $Z$.}
\end{theorem}

\begin{proof}
By Theorems 4.4, 4.15, 4.16, 4.17, and 4.18.
\end{proof}

The definition of possibility formulated in Theorem 4.20 conforms with the intuition. Hence, it follows from Theorem 4.15 that a sentence is possible (i.e. possibly true) if and only if its negation is not necessary (i.e. possibly false) and from Theorem 4.16 that a disjunction of sentences is possible if and only if at least one of the sentences forming the disjunction is possible, a result naturally extending, {\it mutatis mutandis}, to material implication by Theorem 4.18. By Theorem 4.17, however, a conjunction of sentences may not be possible although all the sentences forming the conjunction are possible.

\begin{theorem}
{\it Let $(\mathfrak{A},Z,\mathcal{E},\zeta)$ be a model for modal sentential calculi. For every equivalence class $[F]$ of formulas of $\mathcal{L}$,
\begin{gather*}
\tag{VII}
\textstyle{{\rm int}~{\rm cl}~(\zeta([F])^\ast\triangle\,B)=\emptyset~\text{if and only if}~B\in\zeta([F]),}
\end{gather*}
the symmetric difference of the open-closed representative $\zeta([F])^\ast$ of the image $\zeta([F])$ of $[F]$ with a Baire set $B$ being nowhere dense in $Z$ $($`impossible'\,$)$ if and only if $\zeta([F])^\ast$ and $B$ belong to the same equivalence class of Baire sets.}
\end{theorem}

\begin{proof}
Suppose that $B\in\zeta([F])$. Then $\zeta([F])^\ast\triangle\,B$ is nowhere dense in $Z$ by Theorem 2.11. Conversely, suppose that $\zeta([F])^\ast\triangle\,B$ is nowhere dense in $Z$; it is shown that $B\in\zeta([F])$. For suppose this is not the case. Then there exists an equivalence class $[F']$ of formulas of $\mathcal{L}$ such that $B\in\zeta([F'])$. It follows that $\zeta([F'])^\ast\triangle\,B$ is nowhere dense in $Z$ as well as $\zeta([F])^\ast\triangle\,\zeta([F'])^\ast$ because $\zeta([F])^\ast\triangle\,\zeta([F'])^\ast\subseteq (\zeta([F])^\ast\triangle\,B)\cup (\zeta([F'])^\ast\triangle\,B)$, which contradicts the Baire category theorem for compact Hausdorff spaces whereby the interior of meagre sets, hence, here, that of nowhere dense sets, is empty because $\zeta([F])^\ast\triangle\,\zeta([F'])^\ast$ has a non-empty interior. Hence, $B\in\zeta([F])$ as required.
\end{proof}

At first glance, the reader may contend that Schema (VII) is inconsequential: $\zeta([F])^\ast\triangle\,B$ is nowhere dense in $Z$, that is to say, $\zeta([F])^\ast$ and $B$ are impossible to discriminate from each other, if and only if $\zeta([F])^\ast$ and $B$ belong to the same equivalence class of Baire sets, and the equivalence relation defined on $\mathcal{B}a(Z)$ naturally inducing a congruence relation on $\mathcal{E}$, they are interchangeable, in Leibniz's phraseology, {\it salva veritate}.

But the present state of affairs is more subtle. For recall that Schemas (III)--(VI) are theorems of the metalanguage $\mathcal{L}^E$ of mathematical English if and only if $\zeta([F])^\ast$ is open-closed. Then, although $\zeta([F])^\ast\triangle\,B$ is nowhere dense in $Z$, that is to say, although $\zeta([F])^\ast$ and $B$ are impossible to discriminate from each other, they are {\it not} interchangeable {\it salva veritate}.

Schema (VII) uncovers, therefore, a particularly acute problem: Schemas (III)--(VI), as well as Schemas (I)--(II) for that matter, can be false in the sense that they are not theorems of $\mathcal{L}^E$. Moreover, no remedy is available. For suppose that $\mathcal{L}^{EE}$ is a meta-metalanguage capable of assessing the truth-value of each schema. Then this assessment must be indubitably true to be conclusive. A contradiction because $\mathcal{F}$ is a subalgebra of $\mathcal{E}$.

\begin{definition}
The impossibility to discriminate a true or false or necessary or possibly true or possibly false or impossible sentence which is a theorem of $\mathcal{L}^E$ from a true or false or necessary or possibly true or possibly false or impossible sentence which is not a theorem of $\mathcal{L}^E$ is called {\it agnoia}.\footnote{The word `agnoia' is the romanisation of the Ancient Greek word \textgreek{>'agnoia}, `ignorance'.}
\end{definition}

The systematic study of agnoia in relation to uncertainty modelling is the domain of {\it formal agnoiology}.\footnote{The word `agnoiology' was coined by nineteenth-century Scottish philosopher James Frederick Ferrier \cite{Ferrier54}, who saw in agnoiology, the theory of ignorance, one of the key divisions of metaphysics alongside epistemology and ontology.}

\begin{definition}
A formally correct definition of the predicate {\it is impossible}, formulated in the metalanguage $\mathcal{L}^E$ of mathematical English, is {\it weakly materially adequate} if it implies some instances of Schema (III) plus all instances of Schema (VII).
\end{definition}

\begin{theorem}
{\it Let $(\mathfrak{A},Z,\mathcal{E},\zeta)$ be a model for modal sentential calculi. The following schemas constitute a formally correct and weakly materially adequate definition of the predicate \textup{is impossible:}
\begin{align*}
&\textstyle{{\rm int}~{\rm cl}~\zeta([F])^\ast=\emptyset~\text{if and only if}~\zeta([F])^\ast=\emptyset;}\\
&\textstyle{{\rm int}~{\rm cl}~\zeta(-[F])^\ast=\emptyset~\text{if and only if}~{\rm cl}~{\rm int}~\zeta([F])^\ast=Z;}\\
&\textstyle{{\rm int}~{\rm cl}~\zeta(\bigvee_{n\in\omega}[F_n])^\ast=\emptyset~\text{if and only if}~{\rm int}~{\rm cl}~\zeta([F_n])^\ast=\emptyset~\text{for every}~n;}\\
&\textstyle{{\rm int}~{\rm cl}~\zeta(\bigwedge_{n\in\omega}[F_n])^\ast=\emptyset~\text{whenever}~{\rm int}~{\rm cl}~\zeta([F_n])^\ast=\emptyset~\text{for some}~n;}\\
&\textstyle{{\rm int}~{\rm cl}~\zeta([F_1]\rightarrow [F_2])^\ast=\emptyset~\text{if and only if}}\\
&\hspace{2cm}\textstyle{{\rm int}~{\rm cl}~\zeta(-[F_1])^\ast=\emptyset~\text{and}~{\rm int}~{\rm cl}~\zeta([F_2])^\ast=\emptyset;}\\
&\textstyle{{\rm int}~{\rm cl}~(\zeta([F])^\ast\triangle\,B)=\emptyset~\text{if and only if}~B\in\zeta([F]),}
\end{align*}
the open-closed representative $\zeta([F])^\ast$ of the image $\zeta([F])$ of $[F]$ being nowhere dense in $Z$ if and only if it is the empty set; or $\zeta(-[F])^\ast$ is nowhere dense in $Z$ if and only if $\zeta([F])^\ast$ is non-boundary in $Z$; or $\zeta(\bigvee_{n\in\omega}[F_n])^\ast$ is nowhere dense in $Z$ if and only if all $\zeta([F_0])^\ast,\zeta([F_1])^\ast,\zeta([F_2])^\ast,\ldots$ are nowhere dense in $Z$; or $\zeta(\bigwedge_{n\in\omega}[F_n])^\ast$ is nowhere dense in $Z$ whenever at least one $\zeta([F_0])^\ast,\zeta([F_1])^\ast,\linebreak\zeta([F_2])^\ast,\ldots$ is nowhere dense in $Z$; or $\zeta([F_1]\rightarrow [F_2])^\ast$ is nowhere dense in $Z$ if and only if both $\zeta(-[F_1])^\ast$ and $\zeta([F_2])^\ast$ are nowhere dense in $Z$; or $\zeta([F])^\ast\triangle\,B$ is nowhere dense in $Z$ if and only if $\zeta([F])^\ast$ and $B$ belong to the same equivalence class of Baire sets.}
\end{theorem}

\begin{proof}
By Definition 4.1 and the contraduals of Theorems 4.9, 4.10, 4.11, and 4.12. A glance at Theorem 4.21 completes the proof.
\end{proof}


\section{Conclusion}

This study introduced a class of models for modal sentential calculi which does not require additional operators on the language of classical sentential calculi to capture some notion of necessity, possibility, and impossibility. The question arises as to how valuable this contribution to the literature is. The following initial appraisal can be made:\\
(i) These models are less flexible than those within the Lewis--G\"{o}del paradigm, the logician being deprived of the possibility to select different sets of axioms to describe different systems of modal logic;\\
(ii) These models do not capture Lewis's strict implication but a weak form of material implication, which may impede philosophical studies in which the concept of strict implication plays a central role;\\
(iii) These models do not account for iterated modalities, whose study is deferred to the fourth paper of this series. In fact, iterated modalities can be modelled via a hierarchy of metalanguages (i.e. a hierarchy of set theories).

There are, however, several advantages in adopting these models over those within the Lewis--G\"{o}del paradigm. In particular:\\
(iv) These models simplify the mathematical treatment of modality by defining necessity, possibility, and impossibility directly at the level of the metalanguage instead of adding an operator $\Box$ or $\Diamond$ to the object-language;\\
(v) These models introduce a fundamental, endogenous notion of uncertainty in logic with the concept of agnoia, a semantic property of sentences which makes it impossible to discriminate, say, a true sentence $p$ which is a theorem of the metalanguage from a true sentence $p$ which is not;\\
(vi) These models lay the foundations of an efficient integration of logic and probability by taking into account the necessary conditions for the existence of a strictly positive, countably additive probability measure on a $\sigma$-complete Boolean algebra.

This last aspect of modal sentential calculi has only been touched upon in this study and is the primary subject matter of the second paper of this series on the foundations of logic and probability.


\end{document}